\documentclass[a4paper, 12pt]{article}

\usepackage{amsmath,amsfonts,amssymb,amsthm,amscd}
\usepackage{color}

\title{On compactifications and the topological dynamics of definable groups}

\date{\today}

\author{Jakub Gismatullin\thanks{Supported by the Marie Curie Intra-European Fellowship MODGROUP no. PIEF-GA-2009-254123 and Polish Goverment MNiSW grant N N201 545938}\\University of Leeds and Uniwersytet Wroc\l awski
\and Davide Penazzi\thanks{Supported by a postdoctoral fellowship on EPSRC grant EP/I002294/1}\\University of Leeds
\and Anand Pillay\thanks{Supported
by EPSRC grant EP/I002294/1}\\University of Leeds} 

\theoremstyle{plain}
\newtheorem{Theorem}{Theorem}[section]
\newtheorem{Proposition}[Theorem]{Proposition}
\newtheorem{Lemma}[Theorem]{Lemma}

\newtheorem{Fact}[Theorem]{Fact}

\theoremstyle{definition}
\newtheorem{Definition}[Theorem]{Definition}

\newtheorem{Question}[Theorem]{Question}

\theoremstyle{remark}
\newtheorem{Remark}[Theorem]{Remark}
\newtheorem*{claim}{Claim}
\newtheorem*{acknowledgements}{Acknowledgements}

\newcommand{\R}{\mathbb R}   
\newcommand{\Q}{\mathbb Q}

\newcommand{\F}{\mathbb F}

\DeclareMathOperator{\tp}{tp}
\DeclareMathOperator{\SL}{SL}
\DeclareMathOperator{\Th}{Th}
\DeclareMathOperator{\Homeo}{Homeo}

\begin{document}

\maketitle

\begin{abstract} 
We discuss definable compactifications and topological dynamics. For $G$ a group definable in some structure $M$, we define notions of  ``definable'' compactification of $G$ and ``definable'' action of $G$ on a compact space $X$ (definable $G$-flow), where the latter is under a definability of types assumption on $M$. We describe the universal definable compactification of $G$ as $G^{*}/(G^*)^{00}_{M}$ and the universal definable $G$-ambit as the type space $S_{G}(M)$. We also prove existence and uniqueness of ``universal minimal definable $G$-flows'', and discuss issues of amenability and extreme amenability in this definable category, with a characterization of the latter. For the sake of completeness  we also describe the universal (Bohr) compactification and universal $G$-ambit in model-theoretic terms, when $G$ is a topological group (although it is essentially well-known). 
%
\end{abstract}

\section{Introduction and preliminaries}

Given a topological (Hausdorff) group $G$, a compactification of $G$ is a pair $(H,f)$ where $H$ is a compact topological group and $f\colon G\to H$ a continuous homomorphism with dense image. There is a \emph{universal} such compactification, called the \emph{Bohr compactification}. Let us note immediately that a compactification of the topological group $G$ is a special case of continuous action of $G$ on a compact space $X$, where $X$ has a distinguished point $x_{0}$ with dense orbit under $G$ (a so-called $G$-ambit.) Again there is a universal $G$-ambit.  There is reasonably compehensive account of abstract topological dynamics in \cite{Auslander}. \\

On the other hand, given a group $G$ definable in a structure $M$, a saturated elementary extension $M^*$ of $M$, $G^*$ the interpretation of the formulas defining $G$ in $M^*$ and a type-definable over $M$, normal subgroup $N$ of $G^{*}$ of index at most $2^{|M|}$, $G^{*}/N$, equipped with the so-called \emph{logic topology} is a compact Hausdorff group, and the identity embedding of $G$ in $G^{*}$ induces a homomorphism from $G$ into $G^{*}/N$ with dense image. There is a smallest such $N$ which is called $(G^{*})^{00}_{M}$, a kind of \emph{connected component} of $G^{*}$. More generally we have the action of $G$ on the space $S_{G}(M)$ of complete types over $M$ concentrating on $G$, where $S_{G}(M)$ is the Stone space of the Boolean algebra of definable subsets of $G$. And there is a canonical dense orbit: $G\cdot \tp(1_G/M)$.
\\

Here we will relate these two theories, in the context of two categories: first the classical case of topological groups, and secondly the new case of definable groups. The case of compactifications of topological groups was explicated by Robinson and Hirschfeld (\cite{Robinson}, \cite{Hirschfeld}) in the language of nonstandard analysis. For the more general case of $G$-flows, the explication of the universal $G$-ambit via the \emph{Samuel compactification} of $G$ (\cite{Samuel}, \cite{Uspenskij}), with respect to the right uniformity on $G$, is basically equivalent to the model-theoretic account that we give below.
As far as definable groups are concerned, we give appropriate definitions  (definable compactification, definable $G$-flow), obtaining in a sense a theory of \emph{tame} topological dynamics, although we make a \emph{definability of types} assumption on the model concerned, in the case of group actions. A free group $\F_{n}$ ($n\geq 2$), considered as a first order structure $(\F_{n},\cdot)$, will be extremely amenable as a definable group, although considered as a discrete topological group, it will not be such. \\

A big influence on this paper is work of Newelski (for example \cite{Newelski1}, \cite{Newelski2}) on trying to use the machinery of topological dynamics to extend stable group theory to tame unstable contexts. The present paper is quite soft, aiming partly at putting Newelski's ideas into a formal framework,  \emph{definable, or tame, topological dynamics}.  On the other hand the treatment of $fsg$ groups in \cite{Pillay-fsg}  and the case analysis  of $\SL(2,\R)$ in  \cite{GPP} from this point of view, are rather harder.
\\




In the remainder of this introduction, we recall some very basic model-theoretic notions and constructions.  A good reference for the kind of model theory in the current paper is \cite{Poizat} as well as \cite{NIPI}. Whenever we talk about a topological space we assume Hausdorffness.

We fix a complete theory $T$ in language $L$, a model $M$ of $T$ and a very saturated elementary extension $M^{*}$ of $M$, for example $\kappa$-saturated of cardinality $\kappa$  where $\kappa > 2^{|M|+|L|}$.  For $X$ a definable set in $M^{*}$, definable over $M$  (or even a set definable in $M$),  $S_{X}(M)$ denotes the Stone space of complete types over $M$ which concentrate on (the formula defining) $X$. By a type-definable over $M$ set in $M^{*}$ we mean the common solution set in $M^{*}$ of a collection of formulas over $M$, equivalently an intersection of sets, definable in $M^{*}$ over $M$. Sometimes $L_{M}$ denotes the language $L$ expanded by constants for elements of $M$. \\

We first recall the logic topology on bounded hyperdefinable sets (in a saturated model).

\begin{Definition} \label{def:logic}
Let $X$ be a definable set in the structure $M^*$, definable with parameters from $M$.
\begin{enumerate}
\item[(i)] Suppose $E$ is a type-definable over $M$ equivalence relation with a \emph{bounded} number of classes, that is $<\kappa$ (equivalently $\leq 2^{|M| + |L|}$) many, and $\pi\colon X\to X/E$ the canonical surjection. The \emph{logic topology} on $X/E$ is defined as follows: $Z\subseteq X/E$ is closed if $\pi^{-1}(Z)\subseteq X$ is type-definable over $M$.

\item[(ii)] Suppose $C$ is a compact space, and $f\colon X\to C$ a map. We say that $f$ is {\em definable over $M$}, if for any closed subset $D$ of $C$, $f^{-1}(D)$ is type-definable over $M$.
\end{enumerate}
\end{Definition}

The following is well-known (see e.g. \cite{NIPI}). 

\begin{Lemma} \label{lem:log} Let $X$, $E$ be as in Definition \ref{def:logic} and $C$ be a compact space.
\begin{enumerate}
\item[(i)] The set $X/E$ equipped with the logic topology is a compact space, and $\pi\colon X\to X/E$ is definable over $M$.
\item[(ii)] Conversely, suppose $f\colon X \to C$ is an $M$-definable map. Then:
\begin{enumerate}
\item $D = f(X)$ is closed in $C$,
\item the equivalence relation $E(x,y) \Leftrightarrow f(x) = f(y)$, for $x,y\in X$, is type-definable over $M$ and \emph{bounded},
\item $f$ induces a homeomorphism between $X/E$ equipped with the logic topology and $D$.
\end{enumerate}
\item[(iii)] Let again $f\colon X \to C$ be $M$-definable. Then $f$ factors through the tautological map $a  \mapsto \tp(a/M)$ from $X$ to $S_{X}(M)$, and $C$ is equipped with the quotient topology. 
\end{enumerate}
\end{Lemma}

\begin{Fact} Suppose $G$ is a group definable in $M$ and $G^*$ is the interpretation in $M^*$ of the formulas defining $G$.
\begin{enumerate}
\item[(i)] Let $N$ be a normal, type-definable over $M$, bounded index subgroup of $G^*$. Then $G^*/N$ equipped with the logic topology is a compact group.
\item[(ii)] There is a smallest, type-definable over $M$, subgroup of $G$. It is called $({G^*})^{00}_{M}$, and is normal in $G^*$.
\end{enumerate}
\end{Fact}
 
We also recall:

\begin{Definition}
\begin{enumerate}
\item[(i)] A type $p(x)\in S(M)$ is \emph{definable} if for any $\varphi(x,y)\in L$, $\{b\in M\colon\varphi(x,b)\in p\}$ is a definable set in $M$.
\item[(ii)] If $A\supseteq M$, and $q(x)\in S(A)$, then $q$ is said to be \emph{finitely satisfiable in $M$}  (or equivalently a \emph{coheir of $q|M$}) if every formula in $q$ is satisfied by some element or tuple from $M$.
\item[(iii)] If $A\supseteq M$ and $q(x)\in S(A)$, then $q$ is said to be an \emph{heir of $q|M$} if for any $L_{M}$-formula $\varphi(x,y)$ such that $\varphi(x,a)\in q$ for some $a\in A$, there is $m\in M$ such that $\varphi(x,m)\in q$. 
\end{enumerate}
\end{Definition}

\begin{Fact} 
\begin{enumerate}
\item[(i)] $\tp(a/M,b)$ is a coheir of $\tp(a/M)$ if and only if $\tp(b/M,a)$ is an heir of $\tp(b/M)$.
\item[(ii)] $p(x)\in S(M)$ is definable if and only if $p$ has a unique heir over any $A\supseteq M$.
\end{enumerate}
\end{Fact}

The property, for a given theory $T$, that all complete types over all models are definable (equivalently all complete types over a sufficiently saturated model are definable), is very strong, and equivalent to \emph{stability} of $T$.
On the other hand there are certain important structures $M$ over which all complete types are definable, even though $\Th(M)$ is unstable. Examples are $(\R,+,\cdot)$ and $({\Q}_{p},+,\cdot)$. There is also a class of first order theories, properly containing the class of stable theories, for which certain ``mild'' expansions of arbitrary models have the property that all types over them are definable:  $T$ is (or has) $NIP$, if (working in a saturated model $M^{*}$) for any $L$-formula $\varphi(x,y)$, indiscernible sequence $(a_{i}:i<\omega)$ and $b\in M^{*}$, the truth value of $\varphi(a_{i},b)$ stabilizes as $i\to \infty$. Given a model $M$ of any theory $T$, by $M^{ext}$ we mean the structure whose universe is $M$ and which has predicates for all \emph{externally definable} subsets of $M^{n}$: where $X\subseteq M^{n}$ is said to be externally definable,  if there is a formula $\psi(x)$ possibly with parameters from $M^{*}$ such that $X = \{a\in M^{n} : M^{*}\models \psi(a)\}$. Note that $M^{ext}$ has (essentially) constants for all elements.  The following is due to Shelah \cite{Shelah}, but see also \cite{Pillay} for a simplified account. 

\begin{Fact} Assume that $T$ has $NIP$. Then for any model $M$ of $T$, $\Th(M^{ext})$ has quantifier elimination and $NIP$. Moreover all complete types (in any number of variables) over $M^{ext}$ are definable. 
\end{Fact}

In fact Shelah suggests, in Thesis 0.5 of \cite{dependentdreams}, that in studying an $NIP$ theory $T$, one should from the start work with the class of externally definable sets.  This in some sense provides some justification for our ``definability of types'' assumption when we introduce ``definable topological dynamics'' in Section 3.\\

\section{The topological case}
We give a brief model-theoretic description of the topological case. We refer the reader to \cite{Uspenskij} for a nice treatment of topological dynamics.
\\

Let $G$ be a topological group. As mentioned in the introduction $G$ has a unique ``universal compactification'', called \emph{Bohr compactification}. Here is a summary of a model-theoretic account
. Let $M$ be a structure whose universe is $G$, and has a symbol for the group operation, and, for each open subset of $G$ a predicate symbol $P_{U}$ whose interpretation in $M$ is $U$ (and maybe more). We suppose the language $L$ of $M$ has cardinality at most $2^{|M|}$. Let $M^{*}$ be a very saturated  elementary extension of $M$ and $G^{*}$ the corresponding group. Let $U^{*}$ denote the interpretation of $P_{U}$ in $M^{*}$. The following proposition immediately follows from \cite[Theorem 5.6]{abscon}.

\begin{Proposition} Let  $N$ be the smallest subgroup of $G^{*}$ which has bounded index, and  is an intersection of some of the $U^{*}$'s. Then:
\begin{enumerate}
\item[(i)] $N$ is normal  (hence $G^{*}/N$ with the logic topology is a compact group),
\item[(ii)] The mapping $f$ from $G$ to $G^{*}/N$ induced by the identity embedding of $G$ in $G^{*}$ is continuous and is the universal compactification of the topological group $G$.
\end{enumerate}
\end{Proposition}

\vspace{5mm}
\noindent
Now we consider topological dynamics. $G$ remains a topological group. By a $G$-flow $(X,G)$ we mean a continuous action of $G$ on a compact space $X$  (i.e. an action of $G$ on $X$ such that  the corresponding function $G\times X \to X$ is continuous).  Note that for $X$ a compact space, the space $\Homeo(X)$ of homeomorphisms of $X$ equipped with the compact-open topology, is a compact group. Moreover a $G$-flow $(X,G)$ is precisely given by a continuous homomorphism from $G$ to $\Homeo(X)$. Note that if $x\in X$, then the map taking $g\in G$ to $g\cdot x \in X$ is continuous. 
By a \emph{$G$-ambit} $(X,x_{0},G)$ we mean a $G$-flow $(X,G)$ together with a point $x_{0}\in X$ such that the orbit $G\cdot x_{0}$ is dense in $X$. There is a \emph{universal $G$-ambit}, sometimes called $S(G)$, with distinguished point $e$ say: for every $G$-flow $(X,G)$ and $p\in X$ there is a unique map of $G$-flows from $S(G)$ to $X$ which takes $e$ to $p$. The universal $G$-ambit is unique, by definition. The following model-theoretic (nonstandard analytic) account of the universal $G$-ambit is easily seen to follow from the account of $S(G)$ as the ``Samuel compactification'' of the uniform space $(G,{\mathcal R})$ where $\mathcal R$ is the canonical right uniformity on $G$ (see \cite[Section 2]{Uspenskij}). The basic entourages of $\mathcal R$ are $\{(x,y)\in G^2 : xy^{-1}\in V\}$, where $V\subseteq G$ range over open sets containing the identity $1$ of $G$.

\begin{Proposition} Let $G$ be a topological group, and $M$ a structure as described at the beginning of this section. Let $E$ be the finest bounded, type-definable over $M$ equivalence relation on $G^{*}$ which satisfies the following condition:
\begin{quote}
whenever $g,h\in G^{*}$ and $gh^{-1}\in U^{*}$ for every $U$ which is a neighbourhood of the identity, then $E(g,h)$.
\end{quote}
Then $E$ is $G$-invariant, and  $G^{*}/E$ with the logic topology and distinguished point $1/E$ is the universal $G$-ambit.
\end{Proposition}
\begin{proof}[Proof (sketch)]
$S(G)$ is the completion of $G$ with respect to the finest precompact uniformity $\mathcal U$ which is coarser than $\mathcal R$. Such a uniformity corresponds to a type-definable equivalence relation $E$ on $G^*$. Since $\mathcal U$ is precompact, $E$ has boundedly many classes. Also $\mathcal U$ is $G$-invariant, as $\mathcal R$ is. The completion of $G$ with respect to $\mathcal U$ is exactly $G^{*}/E$ with the logic topology.
\end{proof}

We give a few additional remarks. The universal $G$-ambit $S(G)$ has a canonical semigroup structure (which in the proof of Theorem 2.1 in \cite{Uspenskij} is seen to follow directly from the universality property of $S(G)$). This semigroup structure is in turn used to prove the uniqueness of the universal minimal $G$-flow:  A $G$-flow $(X,G)$ is \emph{minimal} if $X$ has no proper $G$-subflow  (i.e. closed $G$-invariant subset). A universal minimal $G$-flow is a minimal $G$-flow $(X,G)$ such that for any (minimal) $G$-flow $(Y,G)$ there is some map of $G$-flows $f\colon X\to Y$.
So uniqueness of the universal minimal $G$-flow is not built into the definition, but is nevertheless true. Computation of the universal minimal flow $M(G)$ for various topological groups is an important enterprise. We have been told that for no locally compact noncompact topological group $G$, has $M(G)$ been explicitly described. 
The topological group $G$ is said to be (extremely) amenable if for every $G$-flow $(X,G)$, there is a $G$-invariant Borel probability measure on $X$ ($X$ has a $G$-invariant point).  Clearly it suffices for the universal $G$-ambit to have this property. 

\section{The definable case}
The aim in this main section of the paper is to try to give appropriate analogues of compactification and action on a compact space, for definable groups rather than topological groups. Many of our formulations can possibly be  improved or modified. We also ask a few questions.
We fix a first order structure $M$, not necessarily saturated. By a definable set in $M$ we mean a subset $Y$ of $M^{n}$ definable \emph{with parameters from $M$}.  $M^{*}$ denotes a very saturated elementary extension of $M$. For $Y$ a definable set in $M$, we often use $Y^{*}$ to denote the interpretation in $M^{*}$ of the formula which defines $Y$ in $M$.
\\ 


We first give an analogue of Definition \ref{def:logic}(ii) in the nonsaturated case:

\begin{Definition} \label{def:nonsat}
Let $Y$ be a definable set in $M$ and $C$ a compact space. By a {\em definable map} $f$ from $Y$ to $C$ we mean a map $f$ such that for any disjoint closed subsets $C_{1}, C_{2}$ of $C$ there is a definable subset $Y'$ of $Y$ such that $f^{-1}(C_{1})\subseteq Y'$ and $Y'\cap f^{-1}(C_{2}) = \emptyset$. 
\end{Definition}

\begin{Lemma} \label{lem:extend} Suppose $Y$ is a definable set in $M$.
\begin{enumerate}
\item[(i)] Suppose $f\colon Y\to C$ is a definable map from $Y$ to a compact space $C$. Then $f$ extends uniquely to an $M$-definable map from $Y^{*}$ to $C$ (in the sense of Definition \ref{def:logic}(ii)).
\item[(ii)] Conversely, suppose $f^{*}$ is an $M$-definable map from $Y^{*}$ to the compact space $C$ (in the sense of Definition \ref{def:logic}(ii)). Then the restriction $f$ of $f^{*}$ to $Y$ is a definable map from $Y$ to $C$ in the sense of Definition \ref{def:nonsat}.
\end{enumerate}
\end{Lemma}
\begin{proof}
$(i)$ We first define $f^{*}$. Let $c\in Y^{*}$ and let $p(y) = \tp(c/M)$. For $\varphi(y)$ a formula in $p$, let $\overline{f(\varphi(M))}$ denote the closure of $f(\varphi(M))$ in $C$.

\begin{claim}
$\bigcap_{\varphi\in p}\overline{f(\varphi(M))}$ is a singleton in $C$.
\end{claim}
\begin{proof}[Proof of Claim.]  Suppose for a contradiction that  $a\neq b$ are both in $\bigcap_{\varphi\in p}\overline{f(\varphi(M))}$. Let $C_{1}$, $C_{2}$ be disjoint closed neighbourhoods in $C$ of $a,b$ respectively. So there is $\varphi(y)$ over $M$ such that  $f^{-1}(C_{1})\subseteq \varphi(M)$, and $\varphi(M)\cap f^{-1}(C_{2}) = \emptyset$. 
Without loss of generality $\varphi(y)\in p(y)$. Now $f(\varphi(M))$ is disjoint from $C_{2}$ hence (as $C_{2}$ contains an open neighbourhood of $b$), $\overline{f(\varphi(M))}$ does not contain $b$, a contradiction. 
\end{proof}

\vspace{2mm}
\noindent
So define $f^{*}(c)$ to be the unique element in $\bigcap_{\varphi\in p}\overline{f(\varphi(M))}$.
\newline
A similar argument to the claim shows that $f^{*}$ is definable over $M$ (if $D\subseteq C$ is closed, let $\Sigma(y)$ be the set of formulas $\varphi(y)$ over $M$ such that $f^{-1}(D)\subseteq \varphi(M)$, and show that if $c$ satisfies $\Sigma(y)$ in $Y^{*}$ then $f^{*}(c) \in D$).  Uniqueness of $f^{*}$ is also clear.

\vspace{2mm}
\noindent
$(ii)$ Let $C_{1}, C_{2}$ be disjoint closed subsets of $C$. Let $X_{i} = (f^{*})^{-1}(C_{i})$ for $i=1,2$. As $X_{1}$ and $X_{2}$ are disjoint, type-definable over $M$ subsets of $Y^{*}$, by compactness (or saturation of $M^{*}$) they separated by an $M$-definable set $Z$. So $f^{-1}(C_{1})$ and $f^{-1}(C_{2})$ are separated by $Z(M)$, a definable set in $M$.
\end{proof}

We start with the rather easy case of group compactifications.

\begin{Definition} Let $G$ be a group definable in $M$. By a {\em definable compactification} of $G$ we mean a definable homomorphism from $G$ to a compact group $C$ with dense image.
\end{Definition}

\begin{Proposition} Let $G$ be a group definable in $M$. Then there is a (unique) universal definable compactification of $G$, and it is precisely $G^*/(G^*)^{00}_{M}$ (where the homomorphism from $G$ to $G^*/(G^*)^{00}_{M}$ is that induced by the identity embedding of $G$ in $G^{*}$).
\end{Proposition}
\begin{proof} Let $f\colon G\to C$ be a definable compactification of $G$. By Lemma \ref{lem:extend}$(i)$, $f$ lifts uniquely to an $M$-definable map $f^{*}$ from $G^{*}$  onto $C$.  We claim that $f^{*}$ is a group homomorphism.  The proof of Lemma \ref{lem:extend}$(i)$ shows that for any $a\in G^{*}$, $f^{*}(a) = \bigcap_{\varphi\in p}\overline{f(\varphi(M))}$, where $p = \tp(a/M)$.
Let $b\in G^{*}$ and $q = \tp(b/M)$. Then \[\bigcap_{\varphi\in p, \psi\in q}\overline{f(\varphi(M)\cdot\psi(M))} =
\bigcap_{\varphi\in p, \psi\in q}\overline{f(\varphi(M))}\cdot\overline{f(\psi(M))} = 
\bigcap_{\varphi\in p}\overline{f(\varphi(M))}\cdot\bigcap_{\psi\in q}\overline{f(\psi(M))}\] = $f^{*}(a)\cdot f^{*}(b)$. This is 
enough to show that $f^{*}$ is a homomorphism. So the kernel of $f^{*}$ is a normal type-definable (over $M$) subgroup of $G^{*}$ of bounded index. Also  by \ref{lem:log}, the topology on $C$ coincides with the logic topology (on $G^{*}/N$). So if $(G^*)^{00}_{M}$ is the smallest type-definable over $M$ subgroup of $G^{*}$ of bounded index, then we have a canonical surjection map $\pi\colon G^{*}/(G^{*})^{00}_M \to C = G^{*}/N$, which is clearly a continuous homomorphism and commutes with the canonical homomorphisms from $G$. 
\end{proof}

Now for definable group actions.  Let again $G$ be a group definable in a structure $M$. So we have the action of $G$ on the type space $S_{G}(M)$.
$G$, as a subgroup of $G^{*}$, also acts on the left on $G^{*}$. Suppose $E$ is a bounded, type-definable over $M$, $G$-invariant equivalence relation on $G^{*}$. Then $G$ acts on $G^{*}/E$ (a compact space equipped with the logic topology), and this action factors through the action of $G$ on $S_{G}(M)$.  Moreover if $1$ is the identity element of $G$ then the orbit of $1/E$ in $G^{*}/E$ is dense. So, if we simply stipulate that the compact spaces on which $G$ acts ``definably" are simply such quotients $G^{*}/E$, then by definition  there is a universal such action, namely the action of $G$ on $S_{G}(M)$. The notion of ``definable amenability" of $G$ from the papers \cite{NIPI}, \cite{NIPII} for example, then says precisely that there is a $G$-invariant Borel probability measure on $S_{G}(M)$. And one could likewise define $G$ to be extremely amenable if there is a $G$-invariant point in $S_{G}(M)$, namely an invariant type. \\

Now our original motivation was to give an intrinsic description of the actions described in the last paragraph. For a number of reasons, including that of developing a robust theory in analogy with the topological case, we will only do this under a definability of types assumption: all types in $S_{G}(M)$ are definable. As remarked in the introduction important unstable structures such as the field $\mathbb R$ of real numbers, or the field ${\mathbb Q}_{p}$ of $p$-adic numbers, do have the feature that all types over them are definable.  So our theory will apply to real and $p$-adic  semialgebraic Lie groups. 
Given an $NIP$ theory $T$, an arbitrary model $M$ of $T$ and group $G$ definable in $M$, it would be natural to apply the theory we develop to $G$ as a group definable in the expansion $M^{ext}$. In fact the third author already did this successfully for $fsg$ groups in \cite{Pillay-fsg}\\

\noindent
ASSUMPTION. $M$ is a first order structure, $G$ is a group definable in $M$, and all types in $S_{G}(M)$ are definable.

\vspace{2mm}
\noindent
So note we do not assume that ALL complete types over $M$ are definable, just complete types extending the formula $x\in G$. 

\vspace{2mm}
\noindent
We now give a somewhat strong definition of a {\em definable action} of $G$ on a compact space $X$, suitable for our purposes, under the ASSUMPTION above.

\begin{Definition} \label{def:g-flow} Let $X$ be a compact space.
\begin{enumerate}
\item[(i)] By a definable action of $G$ on $X$ (or definable $G$-flow)  we mean that $G$ acts on $X$ by homeomorphisms, and for each $x\in X$, the map $f_{x}\colon G\to X$ taking $g$ to $g\cdot x$ is definable.
\item[(ii)] A definable $G$-ambit is a definable $G$-flow $(X,G)$ with a distinguished point $x_{0}\in X$ such that the orbit $G\cdot x_{0}$ is dense in $X$.
\end{enumerate}
\end{Definition}

\begin{Remark}
\begin{enumerate}
\item[(i)] The second clause in Definition \ref{def:g-flow}$(i)$ is equivalent to saying that the induced map from $G$ to the compact space $X^{X}$ is definable.

\item[(ii)] Itai Ben Yaacov suggested to the third author that the topologically correct definition in (i) is that in addition to $G$ acting by homeomorphisms, the map from $G$ to the space $\Homeo(X)$ of homeomorphisms of $X$, where $\Homeo(X)$ is equipped with the compact-open topology, is definable when considered as a map from $G$ into the Roelcke compactification of $\Homeo(X)$. Under our definability of types assumption, definable $G$-ambits will  have this property.
\end{enumerate}
\end{Remark} 

\begin{Lemma}
\begin{enumerate}
\item[(i)] The action of $G$ on $S_{G}(M)$ is definable.
\item[(ii)] Moreover for any bounded type-definable over $M$, $G$-invariant equivalence relation $E$ on $G^{*}$, the action of $G$ on $G^{*}/E$ (equipped with the logic topology) is definable.
\end{enumerate}
\end{Lemma}
\begin{proof} $(i)$ Clearly $G$ acts on $S_{G}(M)$ by homeomorphisms. As $S_{G}(M)$ is totally disconnected, the second clause in Definition \ref{def:g-flow}$(i)$, says that for any $L_{M}$-formula $\varphi(x)$, and $p(x)\in S_{G}(M)$, $\{g\in G: \varphi(x)\in gp\}$ is a definable subset of $G$. And this follows 
by definability of $p(x)$.

$(ii)$ This follows from $(i)$: Let $X = G^{*}/E$ and $\pi$ the (continuous) $G$-invariant surjection from $S_{G}(M)$ to $X$. Suppose $a\in X$ and $C_{1}$, $C_{2}$ are disjoint closed subsets of $X$. Let $D_{i} = \pi^{-1}(C_{i})$ for $i=1,2$. So $D_{1}$, $D_{2}$ are disjoint closed subsets of $S_{G}(M)$, and by compactness ($E$ being type-definable over $M$), there are $M$-definable sets $D_{1}'$, $D_{2}'$, containing $D_{1}$, $D_{2}$ respectively, such that $\pi(D_{1}')$ and $\pi(D_{2}')$ are disjoint. Let $p(x)\in \pi^{-1}(a)$. So as in (i) $\{g\in G: gp\in D_{1}'\}$ is a definable subset $Y$ of $G$.  We see that $Y$ separates $\{g\in G: g\cdot a \in C_{1}\}$ and $\{g\in G:g\cdot a \in C_{2}\}$, as required.
\end{proof}

\begin{Proposition} \label{prop:g-ambit} There is a (unique) universal definable $G$-ambit, which is precisely the type space $S_{G}(M)$, under the natural action of $G$, where the distinguished element of $S_{G}(M)$ is the identity element $1$ of $G$. Namely for any other definable $G$ ambit $(X,G,x_{0})$ there is a unique continuous (necessarily surjective) map  $h\colon S_{G}(M) \to X$ of definable $G$ flows with $f(1) = x_{0}$.
\end{Proposition}
\begin{proof} Let $(X,G,x_{0})$ be a definable $G$-ambit. Let $f\colon G\to X$ be $f(g) = g\cdot x_{0}$. By Lemma \ref{lem:extend}$(i)$, $f$ extends uniquely to an $M$-definable map $f^{*}$ from $G^{*}$ to $X$, and thus  a continuous map $h\colon S_{G}(M) \to X$. Note that $h$ is surjective as its image is compact and contains $G\cdot x_{0}$. Also note that $h(1) = x_{0}$. It remains to see that $h$ is a map of $G$-spaces. This follows from the construction of $f^{*}$ in the proof of \ref{lem:extend}: for $p\in S_{G}(M)$, and $g\in G$, $gh(p) = g\bigcap_{\varphi\in p}\overline{f(\varphi(M))}$ = 
$\bigcap_{\varphi\in p}g\overline{f(\varphi(M))} = \bigcap_{\varphi\in p}\overline{f(g\varphi(M))}$ = 
$\bigcap_{\psi\in gp}\overline{f(\psi(M))} = h(gp)$ (using the fact that $f(\varphi(M)) = \varphi(M)\cdot x_{0}$).
\end{proof}

We now discuss the semigroup structure on $S_{G}(M)$ 

\begin{Definition}  Let $p, q\in S_{G}(M)$. Then by $p*q$ we mean $\tp(ab/M)$ where $b$ realizes $q$ and $a$ realizes the unique coheir of $p$ over $M,b$.  (Equivalently $a$ realizes $p$ and $b$ realizes the unique heir of $q$ over $M,a$). 
\end{Definition}

So note that if $p = g$ is in $G$ (namely a realized type), then $p*q$ is just $gq$, so the  operation $*\colon S_{G}(M)\times S_{G}(M) \to S_{G}(M)$, extends the group action $G\times X \to X$.  

\begin{Lemma}
\begin{enumerate}
\item[(i)] $*$ is a semigroup operation on $S_{G}(M)$, continuous on the left, namely for given $q\in S_{G}(M)$ the map taking $p$ to $p*q$ is a continuous map from $S_{G}(M)$ to itself. 

\item[(ii)] Given $q\in S_{G}(M)$, let $r_{q}$ be the unique continuous function from $S_{G}(M)$ to itself which takes $1$ to $q$ (given by Proposition \ref{prop:g-ambit}).  Then for any $p\in S_{G}(M)$, $r_{q}(p)$ is precisely $p*q$.
\end{enumerate}
\end{Lemma}
\begin{proof} Left to the reader.
\end{proof}

\begin{Definition}
\begin{enumerate}
\item[(i)] By a {\em minimal definable $G$-flow} we mean a definable $G$-flow $(X,G)$ with no proper closed $G$-invariant subset  (i.e. with no proper $G$-subflow).
\item[(ii)] By a {\em universal minimal definable $G$-flow}, we mean a minimal definable $G$-flow $(I,G)$ such that for any minimal definable $G$-flow $(X,G)$ there is a $G$-map from $I$ to $X$.  (So in fact for every definable $G$-flow $X$, minimal or not, there is a $G$-map from $I$ to $X$).
\end{enumerate}
\end{Definition}

Note first that any minimal $G$-subflow $I$ of $S_{G}(M)$ will be a universal minimal definable $G$-flow. (Because for
any definable $G$-flow $(X,G)$ there is by Proposition \ref{prop:g-ambit} a $G$-map from $S_{G}(M)$ to $X$ whose restriction to $I$ is a $G$ map from $I$ to $X$. )  On the other hand, there is no a priori reason for there to be a unique universal minimal definable $G$-flow  (in the obvious sense). But in fact uniqueness of the the universal minimal definable $G$-flow  \emph{is} true.  This makes use of the semigroup structure on $S_{G}(M)$, and the proof is simply a copy of any of the proofs in the classical situation. We go briefly through the steps, following Section 3 of \cite{Uspenskij}.
Let $I$ be a minimal subflow of $S_{G}(M)$:

\begin{claim}
\begin{enumerate}
\item[(i)] Left ideals (with respect to $*$) of $S_{G}(M)$ coincide with closed $G$-subflows.
\item[(ii)] $I$ contains an idempotent $p_{0}$ and  $q*p_{0} = q$ for all $q\in I$.
\item[(iii)] Every $G$-map $f:I\to I$ has the form $p \to p*q$ for some $q\in I$.
\end{enumerate}
\end{claim}
\begin{proof}
Consider $h = f\circ r_{p_{0}}: S_{G}(M)\to I$. Then $h = r_{t}$ where $h(1) = t\in I$. As $r_{p_{0}}|I$ is the identity (by (ii)), $f$ and $h = r_{t}$ have the same restriction to $I$.
\end{proof}

\begin{claim}
(iv) Every $G$-map $f\colon I \to I$ is a bijection. 
\end{claim}
\begin{proof} By (iii) $f = r_{t}|I$ for some $t\in I$. Now $I*t$ being an ideal contained in $I$ coincides with $I$ (by minimality of $I$) hence $s*t = p_{0}$ for some $s\in I$. Let $g: I \to I$ be the map $r_{s}|I$.  Then $f\circ g = id$. As $f$ was arbitrary all $G$-maps from $I$ to $I$ are surjective. In particular $g$ is hence $f$ is a bijection. 
\end{proof}

\vspace{2mm}
\noindent
We conclude from (iv):
\begin{Proposition}  There is a unique (up to isomorphism of $G$-spaces) universal minimal definable $G$-flow, and it coincides with some (any) minimal $G$-subflow of $S_{G}(M)$. 
\end{Proposition}
\begin{proof} Let $I$ be a given minimal subflow of $S_{G}(M)$, and $J$ any other universal minimal $G$-flow. There are $G$-maps $f:I\to J$ and $g:J\to I$, both necessarily surjective. The composition $g\circ f:I\to I$ is bijective by (iv) above, whereby $f$ is injective, hence a homeomorphism. 
\end{proof}

So part of the content of  \cite{GPP} was to describe the universal minimal definable $G$-flow when $G = SL(2,\R)$ and $M = (\R,+,\cdot)$.

\vspace{2mm}
\noindent
Finally we discuss amenability. Our definability of types assumption remains in place.

\begin{Definition}
\begin{enumerate}
\item[(i)] $G$ is {\em definably amenable} if for every definable $G$-flow $(X,G)$ there is a $G$-invariant Borel probability measure on $X$.
\item[(ii)] $G$ is {\em definably extremely amenable} if for every definable $G$-flow $(X,G)$, $X$ has a fixed point. 
\end{enumerate}
\end{Definition}

\begin{Remark}  By Proposition \ref{prop:g-ambit}, definable (extreme) amenability of $G$ is equivalent to $S_{G}(M)$ having a $G$-invariant Borel probability measure (fixed point)
\end{Remark} 

\begin{Question} Suppose $G$ is definably amenable. Is there then a \emph{definable} $G$-invariant Borel probability measure on $S_{G}(M)$? Where definability of $\mu$ means: for any $L$-formula $\varphi(x,y)$, and closed disjoint subsets $C_{1}, C_{2}$ of $[0,1]$ $\{b\in M: \mu(\varphi(x,b))\in C_{1}\}$, and $\{b\in M:\mu(\varphi(x,b))\in C_{2}\}$ are separated by a definable set in $M$. 
\end{Question}

\begin{Question} \label{ques}  Assume $T$ has $NIP$, $M\models T$ and $G$ is a group definable in $M$ (no definability of types assumption).  
\begin{enumerate}
\item[(i)] Suppose there is a $G$-invariant Borel probability measure on $S_{G}(M)$. Is there also one on $S_{G}(M^{ext})$?
\item[(ii)] Likewise, if there is a fixed point in $S_{G}(M)$ is there also one in $S_{G}(M^{ext})$?
\end{enumerate}
\end{Question}
 
Question \ref{ques}$(i)$ has a positive answer when $T$ is $o$-minimal, essentially by the characterization in \cite{CP} as well as \cite{Pillay-fsg}. We are not sure about (ii), even in the $o$-minimal case.




\vspace{2mm}
\noindent
Finally we return to our definability of types assumption on $S_{G}(M)$, and point out that Pestov's characterization (\cite{Pestov}, Theorem 8.1) of extreme amenability passes suitably to the definable category. We say that a definable subset $Y$ of $G$ is \emph{left generic} if finitely many left translates of $Y$ cover $G$.

\begin{Proposition}  $G$ is (definably) extremely amenable if and only if for each definable left generic subset $Y$ of $G$, 
$YY^{-1} = G$
\end{Proposition}
\begin{proof} Suppose first that there is a $G$-invariant type $p\in S_{G}(M)$. We will prove the right hand side  without assuming any definability of types assumption. Let $X$ be a left generic definable subset of $G$. So for some $g\in G$, $gX\in p$. By $G$-invariance of $p$, $X\in p$. As $p$ is $G$-invariant, $XX^{-1} = G$.
\newline
Now for the converse.   Given a definable subset $Y$ of $G$ we feel free to identify $Y$ with the clopen subset of $S_{G}(M)$ (and of any closed subset of $S_{G}(M)$) it defines. So for $p\in S_{G}(M)$ we may write $p\in Y$, although this corresponds to the formula defining $Y$ being in the type $p$. 
Assume the right hand side. 
Let $\cal M$ be a minimal closed $G$-invariant subset of $S_{G}(M)$.
We will prove that $\cal M$ is a singleton, by showing that for any $p\in \cal M$, and $g \in G$, $gp = p$. Suppose for a contradiction that for some $g\in G$ and $p\in S_{G}(M)$, $gp \neq p$. Then there is a formula (i.e. definable subset of $G$) $Y\in p$ such that $Y\cap g(Y) = \emptyset$.  Note that the union $U$ of all $gY$ must cover $\cal M$, Because if not then ${\cal M} \setminus U$ is closed $G$-invariant and nonempty, contradicting minimality of $\cal M$. By compactness of $\cal M$, it is covered by $g_{1}Y \cup \ldots \cup g_{k}Y$ for some $g_{1},\ldots,g_{k}\in G$.
Let $Y_{1} = \{h\in G: Y\in h(p)\}$. By definability of $p$, $Y_{1}$ is a definable subset of $G$.

\begin{claim}[I]
$g_{1}Y_{1} \cup \ldots \cup g_{k}Y_{1} = G$  (so $Y_{1}$ is left generic in $G$).
\end{claim}
\begin{proof}[Proof of Claim (I)]
Let $h\in G$. So $hp \in {\cal M}$ so $hp\in g_{i}Y$ for some $i=1,\ldots,k$. Hence $g_{i}^{-1}h \in Y_{1}$, and so $h\in g_{i}Y_{1}$. 
\end{proof}

\begin{claim}[II]
$Y_{1}\cap gY_{1} = \emptyset$  (so $Y_{1}Y_{1}^{-1} \neq G$). 
\end{claim}
\begin{proof}[Proof of Claim (II)]
If not, let $h\in Y_{1}$ such that $gh\in Y_{1}$. So $hp\in Y$ and $ghp\in Y$ hence $hp\in Y\cap g^{-1}Y$ implying that $gY\cap Y\neq \emptyset$, contradicting choice of $Y$.
\end{proof}
Claims (I) and (II) contradict the right hand side, completing the proof. 
\end{proof}

\begin{Remark} (i) The proof above (which was adapted from Pestov's proof of Theorem 7.1 in \cite{Pestov}) generalizes to show that the kernel of the action of $G$ on its minimal subflow is precisely the intersection of the sets $YY^{-1}$ for $Y\subseteq G$ definable and left generic.
\newline
(ii) If $Th(M)$ is stable, then $G$ is definably extremely amenable iff $G$ is connected, in which case there is a unique invariant type. 
\newline
(iii) One might try to generalize Proposition 3.17 to the general situation with no definability of types assumption. As an example the following holds:
Every minimal subflow of $S_{G}(M)$ is a singleton if and only if whenever $Y\subseteq G$ is definable and (as a clopen subset of $S_{G}(M)$) meets some minimal subflow of $S_{G}(M)$ then $YY^{-1} = G$. 
\end{Remark}

\begin{acknowledgements}
The first author would like to thank Waldemar Hebisch for many interesting conversations. The third author would like to thank Itai Ben Yaacov and Ehud Hrushovski for helpful comments and suggestions following a talk he gave on this topic in Lyon, in October 2012.
\end{acknowledgements}

\end{document}